\documentclass[reqno,12pt]{amsart}
\usepackage{bbm}
\usepackage[all,pdf]{xy}
\usepackage{epsfig}
\usepackage{amsmath}
\usepackage{amssymb}
\usepackage{amscd}
\usepackage{graphicx}
\usepackage[colorlinks=true]{hyperref}
\hypersetup{urlcolor=red, citecolor=blue}
\makeatletter
\@namedef{subjclassname@2020}{%
	\textup{2020} Mathematics Subject Classification}

\newtheorem{thm}{Theorem}[section]
\newtheorem{lem}[thm]{Lemma}

\newtheorem{cl}[thm]{Claim}

\newtheorem{cor}[thm]{Corollary}

\newtheorem{rem}[thm]{Remark}

\def \N {\mathbb N}

\def \Z {\mathbb Z}

\def \M {\mathcal M}

\numberwithin{equation}{section}

\begin{document}
	\title[minimal frequently stable is almost automorphic]{minimal frequently stable is almost automorphic}

	\author{Leiye Xu and Zongrui Hu}
\address[L. Xu and Z. Hu]{Department of Mathematics, University of Science and Technology of
	China, Hefei, Anhui, 230026, P.R. China}

\email{leoasa@mail.ustc.edu.cn, zongrui@mail.ustc.edu.cn}



\thanks{}
\begin{abstract} We show  that a minimal toplogical dynamical system that is frequently stable  if and only if it is
	almost automorphic.
\end{abstract}

\maketitle
\section{Introduction}

	Throughout this paper, 	a topological dynamical system (t.d.s.) is a pair $(X, T )$ consisting of a compact metric space $X$ and a continuous transformation $T: X\to X$. Equicontinuous systems are known to have simple dynamical behaviours. By the well
known Halmos-von Neumann theorem, a transitive equicontinuous system is conjugate to
a minimal rotation on a compact abelian metric group  \cite{HV}. The notion of mean equicontinuity was introduced
by Li, Tu and Ye \cite{LTY} (which is equivalent to the notion of mean-L-stability in the early paper of
Fomin \cite{F}). It was proved \cite{LTY} that a minimal mean equicontinuous system has discrete spectrum. We refer to the survey \cite{LYY} for more related subjects.

In the study of mean equicontinuous t.d.s., Garc\'ia-Ramos, J\"ager and Ye introduced the notion of  frequently stable (see Section \ref{s-4} for definition). In \cite[Theorem 3.4]{GJY},
 Garc\'ia-Ramos, J\"ager and  Ye proved that a minimal mean equicontinuous t.d.s. is almost automorphic if and only if it is  frequently stable. They proposed a question:
	Does there exist a minimal t.d.s. that is frequently stable but not
	almost automorphic? In this paper, we give a negative answer.
\begin{thm}\label{thm-B}
	Let $(X, T)$ be a minimal t.d.s. Then $(X, T)$ is frequently stable  if and only if it is almost automorphic.
\end{thm}

	 The structure of the paper is as follows. In Section \ref{s-2}, we recall some basic notions and results. In Section \ref{s-4}, we prove Theorem \ref{thm-B}.
\section{Preliminaries}\label{s-2}
Throughout this paper, we denote by $\Z_+$ and $\N$ the sets of non-negative integers
and natural numbers, respectively. We denote the cardinality of a set $A$ by $|A|$.
\subsection{Subsets of $\Z_+$} Let $F$ be a subset of $\Z_+$. The upper density, upper
Banach density and lower
Banach density of $F$ are defined by
$$\overline{D}(F)=\limsup_{N\to\infty}\frac{|F\cap[0,N-1]|}{N}$$
$$\overline{BD}(F)=\limsup_{N-M\to\infty}\frac{|F\cap[M,N-1]|}{M-N}$$
$$\underline{BD}(F)=\liminf_{N-M\to\infty}\frac{|F\cap[M,N-1]|}{M-N}$$
respectively. It is clear that $\underline{BD}(F)\le\overline{D}(F) \le \overline{BD}(F)$ for any $F\subset \Z_+$. When a set
is denoted with braces, for example $\{A\}$, we simply write $\underline{BD}\{A\} =\underline{BD}(\{A\})$.

 \subsection{Topological dynamical systems}	 A t.d.s. $(X,T)$ is called minimal if $X$ contains no proper non-empty closed invariant
 subsets.  We denote the forward orbit of $x\in X$ by
 $$\text{Orb}(x, T ) =\{x, T x,\cdots\}$$
 and its orbit closure by $\overline{\text{orb}(x, T )}$. It is easy to verify that a t.d.s. is minimal if and only if every orbit is dense.
  A factor map $\pi: X\to Y$ between two t.d.s. $(X,T)$ and $(Y,S)$ is a continuous onto map
 which intertwines the actions; we say that $(Y,S)$ is a factor of $(X,T)$ and that $(X,T)$ is an
 extension of $(Y,S)$. A factor map $\pi: (X,T)\to (Y,S)$ is almost $1$-$1$
 if $\{x\in X :|\pi^{-1}\pi(x)=\{x\}\}$ is residual (that is, it is the countable intersection of
 dense open sets).

 A t.d.s. $(X,T)$ is said to be equicontinuous if for any $\epsilon > 0$, there is a $\delta> 0$ such that whenever
 $x, y\in X$ with $d(x, y) < \delta$, then $d(T^nx,T^ny) < \epsilon$ for all $n\in \Z_+$.  There is
 a smallest invariant equivalence relation $S_{eq}$ such that the quotient system $(X_{eq}=X/S_{eq},T)$ is
 equicontinuous \cite{EG}. The equivalence relation $S_{eq}$ is called the equicontinuous structure
 relation and the factor $(X/S_{eq},T)$ is called the maximal equicontinuous factor of
 $(X,T)$.  For a t.d.s. $(X,T)$, denote $(X_{eq}, T_{eq})$ as the  maximal equicontinuous factor  of $(X,T)$ and
 $\pi_{eq}: X\to X_{eq}$ as the factor map.
 
Now we introduce the  notion of sensitive set (S-set).  One can see \cite{YZ} for more information. Given a t.d.s. $(X,T)$. $ (x_i)_{i=1}^n\in X^n$ is a sensitive $n$-tuple if $ (x_i)_{i=1}^n$
 is not on the diagonal $\Delta^n(X)$, and for any open
 neighborhood $U_i$ of $x_i$ and any non-empty open subset $U$ of $X$, there
 is $k\ge 0$ such that $U\cap T^{-k}(U_i)\neq\emptyset$ for $i=1, 2,\cdots,n$. Let $K$ be a subset of $X$ with $|K|\ge 2$. We say $K$ is a sensitive set if for each $ (x_i)_{i=1}^n\in K^n\setminus \Delta^n(X)$, each
 neighbourhood $U_i$ of $x_i$ and each neighbourhood $U$ of $x\in X$ there are $k \in\N$  and $x_i\in U$
 with $T^kx_i\in  U_i$ for each $1\le i\le n$.   The following is from \cite[Theorem 8.2]{HLY}. 
 
 \begin{thm}\label{thm-s}
 	Let $(X, T )$ be a minimal t.d.s. and $\pi_{eq}: X\to X_{eq}$ be the factor
 	map to the maximal equicontinuous factor. Then:
 	\begin{itemize}
 		\item[(1)]Each S-set is contained in some $\pi_{eq}^{-1}y$ for some $y \in X_{eq}$.
 		\item[(2)]For each $y\in X_{eq}$, every $A\subset \pi^{-1}y$ with $|A|\ge 2$ is an S-set.
 	\end{itemize}
 \end{thm}
 
\subsection{Hyperspace system} Let $X$ be a compact Hausdorff topological space. Let $2^X$ be the set of nonempty
closed subsets of $X$ endowed with the Hausdorff metric.
Let $d$ be the metric on $X$, then the Hausdorff metric $d_H$ on $2^X$ may define as follows:
$$d_H(A, B)
= \max\{\max_{a\in A}
	d(a, B),\max_{b\in B}
	d(b, A)\},$$
where $d(x, A) = \inf_{y\in A} d(x, y)$. Then $(2^X,d_H)$ is a compact metric space. 
Let $\{A_i\}_{i=1}^\infty$ be an arbitrary sequence of subsets of X. 
We say that $\{A_i\}_{i=1}^\infty$
 converges to $A$, denoted by $\lim_{i\to\infty} A_i = A$, if
$$\lim_{i\to\infty}d_H(A_i,A)=0.$$

Let $(X, T)$ be a t.d.s. We can induce a system on $2^X$. The action of $T$ on $2^X$ is given
by $TA = \{Tx : x\in A\}$ for each $A\in 2^X$. Then $(2^X, T)$ is a t.d.s. and it is called
the hyperspace system.

\subsection{Probability measure space}
Let $X$ be a compact metric space. Denote by $\mathcal{B}_X$ the Borel $\sigma$-algebra of $X$ and $\mathcal{M}(X)$ be the set of all Borel probability measures on $X$. For $\mu\in \mathcal{M}(X)$,  denote by $\mathcal{B}_X^\mu$
the completion of $\mathcal{B}_X$ under $\mu$ and denote by $\text{supp} \mu$ the support of $\mu$, i.e. the smallest closed subset of $X$ with full measure. In the weak$^*$ topology, $\mathcal{M}(X)$ is a nonempty compact convex space.

Let $(X,T)$ be  a t.d.s. We say $\mu \in \mathcal{M}(X)$ is $T$-invariant if $\mu (T^{-1} B) = \mu (B)$ holds for all $B \in \mathcal{B}_X^\mu$. Denote by $\mathcal{M}(X,T)$ the set of  $T$-invariant Borel probability measures of $(X,T)$.
We say  $\mu \in \mathcal{M}(X,T)$ is ergodic if for any $T$-invariant Borel set $B \in \mathcal{B}_X^\mu$, $\mu (B) = 0$ or $\mu(B) = 1$ holds. Denote by $\mathcal{M}^e(X,T)$ the set of ergodic measures of $(X,T)$.

\subsection{F{\o}lner sequence and generic points}
A sequence of finite nonempty subsets $\{F_n \}_{n=1}^{\infty} $ of $\Z_+$ is called a \emph{F{\o}lner sequence} if $\lim_{n \to +\infty} \frac {|(F_n+m)\Delta F_n|}{|F_n|} = 0$ for every $m \in \Z_+$, where $F+m=\{i+m:i\in F\}$. 
	A F{\o}lner sequence $\{F_n \}_{n=1}^{\infty} $ of $\Z_+$ is \emph{tempered} if  there exists some $C > 0$ such that 
\begin{equation*}
| \bigcup_{k < n} (F_n-F_k) | \le C | F_n |\text{ for all } n \in \mathbb{N}
\end{equation*}
where $E-F=\{e-f: e\in E,f\in F\}$. 

For every F{\o}lner sequence of $\Z_+$, there is a subsequence which is tempered. 
Similar to Birkhoff pointwise ergodic theorem, Lindenstrauss estabilished pointwise ergodic theorem for tempered F{\o}lner sequences \cite{EL}. 
\begin{thm}\label{thm-l}
	Let $(X,T)$ be a t.d.s.  and $\mu$ be a $T$-invariant Borel probability measure. Let $\{F_n \}_{n=1}^{\infty} $ be a tempered F{\o}lner sequence of $\Z_+$. Then for any $f \in L^1(X,\mu)$, there is a $T$-invariant $f^* \in L^1(X,\mu)$ such that 
	\begin{equation*}
	\lim_{n \to +\infty} \frac 1{|F_n|} \sum_{i\in F_n} f (T^ix) = f^* (x) \ \ a.e.
	\end{equation*} 
	In particular, if $\mu$ is ergodic, one has
	\begin{equation*}
	\lim_{n \to +\infty} \frac 1{|F_n|} \sum_{i\in F_n} f (T^ix) = \int f(x) d\mu (x) \ \ a.e.
	\end{equation*} 
\end{thm}

 Let $(X, T)$ be a
t.d.s., $\{F_n\}_{n\ge 1}$ be  a F{\o}lner sequence of $\mathbb{Z}_+$ and $\mu\in \M(X,T)$. We
say that $x_0\in X$ is generic for $\mu$ along $\{F_n\}_{n\ge 1}$ if
$$\frac{1}{|F_n|}\sum_{i\in F_n}\delta_{T^i x_0}\to\mu \text{ weakly* as }n\to\infty$$
where $\delta_x$ is the Dirac mass at $x$. This is equivalent to that for all $f\in C(X)$,
$$\lim_{n\to\infty}\frac{1}{|F_n|}\sum_{i\in F_n}f(T^i x_0)=\int fd \mu.$$
By Theorem \ref{thm-l}, it is easy to see the following corollary.
\begin{cor}\label{cor-1}
Let $(X,T)$ be a t.d.s.  and $\mu$ be an ergodic $T$-invariant Borel probability measure. Let $ \{F_n \}_{n=1}^{\infty} $ be a tempered F{\o}lner sequence of $\Z_+$. Then the set of generic points for $\mu$ along  $ \{F_n \}_{n=1}^{\infty} $ is a
Borel subset of $X$ with $\mu$ measure $1$.
\end{cor}

\section{Proof of Theorem \ref{thm-B}}\label{s-4}
In this section, we prove Theorem \ref{thm-B}. We firstly review some notions about frequently stable.
Let $(X, T)$ be a t.d.s. For $x\in X$ and $\delta>0$, denote $B_\delta(x)=\{y\in X:d(x,y)<\delta\}$. A point $x\in X$ is said to be frequently stable
if for every $\epsilon>0$ there is $\delta > 0$ such that
$$\overline{D}\{i\in \Z_+ : \text{diam}(T^iB_\delta(x)) > \epsilon\} < 1.$$
$(X, T)$ is a frequently stable t.d.s. if all points of $X$ are frequently stable.  For a t.d.s. $(X,T)$, denote $(X_{eq}, T_{eq})$ as the  maximal equicontinuous factor  of $(X,T)$ and
 $\pi_{eq}: X\to X_{eq}$ as the factor map. $(X, T)$ is said to be almost automorphic if $\pi_{eq}$ is 
 almost $1$-$1$. If $(X, T )$ is minimal,
then $(X_{eq}, T_{eq})$ is both minimal and uniquely ergodic, and we denote its unique
invariant measure by $\nu_{eq}$.

Now we are going to  prove Theorem \ref{thm-B}.
In fact, we get more general result. 
\begin{thm}\label{thm-C}Let (X, T) be a minimal t.d.s.
	and
	$\pi_{eq}: (X, T) \to (X_{eq}, T_{eq})$. Then the following are equivalent:
	\begin{itemize}
		\item[(1)] $\pi_{eq}$ is almost $1$-$1$;
		\item[(2)]For every $\epsilon>0$ and $x\in X$ there is $\delta > 0$ such that $\overline{BD}\{i\in \Z_+ : \text{diam}(T^iB_\delta(x)) > \epsilon\} < 1$;
		\item[(3)] There exists $x\in X$ such that for every $\epsilon>0$ there is $\delta > 0$ such that $\underline{BD}\{i\in \Z_+ : \text{diam}(T^iB_\delta(x)) > \epsilon\} < 1$.
	\end{itemize}
\end{thm}
To prove Theorem \ref{thm-C}, we need some lemmas.
\begin{lem}\label{l-1} Let $(X,T)$ be a t.d.s. and  $\{B_i\}_{i\in\N}$ be a sequence of  $2^X$. Let $N\in\N$. If $X=\bigcup_{n=0}^{N-1}T^{-n}\left(B_i\right)$ for $i\ge 1$  and $\lim_{i\to\infty}B_i=B_0$, then $X=\bigcup_{n=0}^{N-1}T^{-n}\left(B_0\right)$.
\end{lem}
\begin{proof} Given $x\in X$. 
	Since $ X=\bigcup_{n=0}^{N-1}T^{-n}\left(B_i\right)$, one can find $n_i\in\{0,1,\cdots,N-1\}$ such that 
	$$x\in T^{-n_i}B_i, i\in\N.$$
	Passing by a subsequence, we can assume that $n_i=n_*\in\{0,1,\cdots,N-1\}$. This implies
	$$T^{n_*}x\in B_i,\text{for }i\in\N.$$
	Since $\lim_{i\to\infty}B_i=B_0$, one has  $T^{n_*}x\in B_0$. Hence $x\in T^{-n_*}\left(B_0\right)\subset \bigcup_{n=0}^{N-1}T^{-n}\left(B_0\right)$. By arbitrariness, $X=\bigcup_{n=0}^{N-1}T^{-n}\left(B_0\right)$.
\end{proof}
\begin{lem}\label{l-2} Let $(X,T)$ be a minimal t.d.s. and  $E\in 2^X$. If $X=\bigcup_{n=0}^{\infty}T^{-n}E$, then the interior of $E$ is not empty.
\end{lem}
\begin{proof}By Baire Category Theorem, there is $n\in\ N$ such that the interior of $T^{-n}E$ is not empty. Since $(X,T)$ is minimal, $T$ is semi open \cite[Theorem 2.5]{S}. And then the interior of $E$ is not empty.
	\end{proof}
Now we prove Theorem \ref{thm-C}.
\begin{proof}(2)$\Rightarrow$(3) is trivial. The proof of (1)$\Rightarrow$(2) is from \cite[Theorem 3.4]{GJY}.  
	For completeness, we repeat it. By hypothesis, there
	exists $y_0\in X_{eq}$ such that $|\pi_{eq}^{-1}(y_0)| = 1$. Let $\epsilon > 0$. There exists $\eta > 0$ such
	that
	$$\text{diam}(\pi_{eq}^{-1}(B_\eta(y))\le\epsilon$$
	for every $y\in B_\eta(y_0)$. Since $\nu_{eq}$ is fully supported, $B_\eta(y_0)$ has positive measure and (using standard arguments) must contain a smaller ball with positive
	measure and $\nu_{eq}$-null boundary. Thus, by strict ergodicity, we have that
	$$\underline{BD}\{n \in N : T_{eq}^ny\in B_\eta(y_0)\} > 0$$
	for every $y\in X_{eq}$.
	Given $x\in X$. Since $\pi_{eq}$ is continuous, there exists $\delta > 0$ such that
$\text{diam}(\pi_{eq}(B_\delta(x)))<\eta.$
	Without loss of generality, we may assume $T_{eq}$ is an isometry. This implies that
	$$B_\eta(T^n_{eq}
	\pi_{eq}(x)) = T^n_{eq}
	B_\eta(\pi_{eq}^{-1}(x))$$
	for every $n\in\N$. Consequently
	$$T^n(B_\delta(x))\subset\pi_{eq}^{-1}(B_\eta(T^n_{eq}\pi_{eq}(x)))$$
	for all $n\in \N$. So, if $T^n_{eq}\pi_{eq}(x)\in B_\eta(y_0)$ then
	$\text{diam}(T^nB_\delta(x)) < \epsilon$;
	since this happens with positive lower Banach density, we conclude  $\overline{BD}\{i\in \Z_+ : \text{diam}(T^iB_\delta(x)) > \epsilon\} < 1$ for $x\in X$.
	
	(3)$\Rightarrow$(1). Assume that $\pi_{eq}$ is not almost $1$-$1$. Then 
	\begin{align}\label{eq-11}\epsilon_0:=\inf_{y\in X_{eq}}\text{diam}(\pi^{-1}(y))>0.
	\end{align}
	Given $0<\epsilon_1<\epsilon_2<\epsilon_0$.  By hypothesis,  there is $x_*\in X$ and $\delta>0$ such that 
	$$\underline{BD}\{i\in \Z_+ : \text{diam}(T^i\overline{B_\delta(x_*)}) > \epsilon_1\} < 1.$$
	There is two sequences $\{N_i\}_{i=1}^\infty,\{M_i\}_{i=1}^\infty$ of $\N$ with $\lim_{i\to\infty}(N_i-M_i)=\infty$ such that 
\begin{align}\label{eq-0}\lim_{i\to\infty}\frac{|\{n\in \Z_+ : \text{diam}(T^n\overline{B_\delta(x_*)}) \le\epsilon_1\}\cap[M_i,N_i-1]|}{N_i-M_i}>0.
\end{align}
Passing by a subsequence, assume
$$\lim_{i\to\infty}\frac{1}{N_i-M_i}\sum_{n=M_i}^{N_i-1}\delta_{T^n\overline{B_\delta(x_*)}}:=\nu\in\mathcal{M}(2^X,T).$$
Denote $\mathcal{E}_{\le \epsilon_1}:=\{E\in 2^X:\text{diam}(E)\le \epsilon_1\}$ and $\mathcal{E}_{<  \epsilon_2}:=\{E\in 2^X:\text{diam}(E)<\epsilon_2\}$.  $\mathcal{E}_{\le \epsilon_1}$ is a closed set and $\mathcal{E}_{< \epsilon_2}$ is an open set. Both of them are Borel measurable. By\eqref{eq-0}, $\nu(\mathcal{E}_{\le \epsilon_1})>0$. There is an ergodic $\nu_e$ with $\text{supp}\nu_e\subset \text{supp}\nu$ such that \begin{align}\label{1}\nu_e(\mathcal{E}_{<\epsilon_2})\ge \nu_e(\mathcal{E}_{\le \epsilon_1})>0.
\end{align} Denote 
$$\mathcal{E}_{\ge \epsilon_2}=\{E\in 2^X:\text{diam}(E)\ge \epsilon_2 \}.$$
It is clear that $\mathcal{E}_{\ge \epsilon_2}$ is a closed subset of $2^X$. Next we are going to show that $\nu(\mathcal{E}_{\ge \epsilon_2})=1$. This conflicts with \eqref{1} and then $\pi_{eq}$ is almost $1$-$1$.

For $x\in X$ and $\gamma>0$, denote 
$$\mathcal{C}(x,\gamma)=\{E\in 2^X:\overline{B_\gamma(x)}\subset E\}.$$

It is clear that each $\mathcal{C}(x,\gamma)$  is a closed subset of $2^X$ and then is Borel measurable.  We have the following claim.
\begin{cl}\label{c-1}There is $x_0\in X$ and $\gamma_0>0$ such that $\nu_e(\mathcal{C}(x_0,\gamma_0))>0$.
\end{cl}
\begin{proof}For $\gamma>0$, denote 
	$$\mathcal{C}(\gamma)=\{E\in 2^X:\text{there is }x\in X \text{ such that }\overline{B_\gamma(x)}\subset E\}.$$ It is clear that $\mathcal{C}(\gamma)$ is a compact subset of $2^X$ and then Borel measurable. Since $(X,T)$ is minimal, there exists $N\in\N$ such that
	\begin{align}\label{eq-3}X=\bigcup_{n=0}^{N-1}T^{-n}\left(T^k\overline{B_\delta(x_*)}\right)\text{ for }k\in\Z_+.
	\end{align}  Notice that $$\text{supp}\nu_e\subset\overline{\text{Orb}(\overline{B_\delta(x_*)},T)}.$$ By \eqref{eq-3} and  Lemma \ref{l-1}, $X=\bigcup_{n=0}^{N-1}T^{-n}\left(E\right)$ for all $E\in \text{supp}\nu_e$. By Lemma \ref{l-2}, for any $E\in \text{supp}\nu_e$, the interior of $E$ is not empty. Hence 
	$$\text{supp}\nu_e\subset \bigcup_{\gamma>0}\mathcal{C}(\gamma).$$
	We can find $\gamma>0$ such that $\nu_e(\mathcal{C}(\gamma))>0$. Put $\gamma_0=\frac{\gamma}{2}$. Let $F$ be a finite and $\gamma_0$-dense subset of $X$. Then $X=\bigcup_{x\in F}B_{\gamma_0}(x)$. 
	For $E\in \mathcal{C}(\gamma)$, there is $x_E\in X$ such that $\overline{B_\gamma(x_E)}\subset E$. There is $x_E'\in F$ such that $x_E\in B_{\gamma_0}(x_E')$. This implies $\overline{B_{\gamma_0}(x_E')}\subset\overline{B_\gamma(x_E)}\subset E$. Therefore,
	$\mathcal{C}(\gamma)\subset \bigcup_{x\in F}\mathcal{C}(x,\gamma_0)$.
	Hence there is $x_0\in F$ such that $\nu_e(\mathcal{C}(x_0,\gamma_0))>0$.
	\end{proof}
Let $x_0\in X$ and $\gamma_0>0$ be as in Claim \ref{c-1}.
\begin{cl}\label{c-2}There exists two sequences $\{\tilde N_i\}_{i=1}^\infty$, $\{\tilde M_i\}_{i=1}^\infty$ of  $\N$ with $\lim_{i\to\infty}(\tilde N_i-\tilde M_i)=\infty$ such that 
		$$\lim_{i\to\infty}\frac{|\{n\in \Z_+ : \text{diam}(T^n\overline{B_{\gamma_0}(x_0)}) \ge \epsilon_2\}\cap[\tilde M_i,\tilde N_i-1]|}{\tilde N_i-\tilde M_i}=1.$$
\end{cl}
	\begin{proof}Given $y\in X_{eq}$. Notice that $\pi^{-1}(T_{eq}^ny)=T^n(\pi^{-1} y)$ for each $n\in\N$. For $i\in\N$, we can find $k_i\in\N$ and a finite subset $\{x^{(i)}\}_{i=1}^{k_i}$ of $\pi_{eq}^{-1}(y)$ such that 
		$T^n\{x^{(i)}_1,\cdots,x^{(i)}_{k_i}\}$ is $\frac{1}{2i}$ dense in $\pi^{-1}(T_{eq}^n y)$ for $n=0,1,\cdots,i$. There is $\xi_i>0$ such that 
		$$x,\tilde x\in X, d(x,\tilde x)<\xi_i\text{ implies }d(T^nx,T^nx')<\frac{1}{2i}\text{ for }n=0,1,\cdots,i.$$ 
		By Theorem \ref{thm-s}, $\{x^{(i)}\}_{i=1}^{k_i}$ is an S-set.  There is $n_i\in\N$ and $\{\tilde x^{(i)}_k\}_{k=1}^{k_i}\subset B_{\gamma_0}(x_0)$ such that 
		$$d(x^{(i)}_k,T^{n_i}\tilde x^{(i)}_k)<\xi_i\text{ for }k=1,2,\cdots,k_i.$$
		This implies
			$$d(T^nx^{(i)}_k,T^{n+n_i}\tilde x^{(i)}_k)<\frac{1}{2i}\text{ for }k=1,2,\cdots,k_i\text{ and }n=0,1,\cdots,i.$$
Since 	$T^n\{x^{(i)}_1,\cdots,x^{(i)}_{k_i}\}$ is $\frac{1}{2i}$ dense in $\pi^{-1}(T_{eq}^n y)$, one has $$\pi^{-1}(T_{eq}^ny)\subset \bigcup_{k=1}^{k_i}B_{\frac{1}{i}}(T^{n+n_i}\tilde x^{(i)}_k).$$ Hence, for $n=0,1,\cdots,i$,
		$$\text{diam}(T^{n_i+n}\overline{B_{\gamma_0}(x_0)})\ge \text{diam}(T^n(\{T^{n_i}\tilde x^{(i)}_1,\cdots,\tilde T^{n_i}x^{(i)}_{k_i}\}))\ge \epsilon_0-\frac{1}{i}.$$
		Denote $\tilde M_i=n_i$ and $\tilde N_i=n_i+i$ for $i\in\N$. Since $\epsilon_2\in(0,\epsilon_0)$, one has 
			$$\lim_{i\to\infty}\frac{|\{n\in \Z_+ : \text{diam}(T^n\overline{B_{\gamma_0}(x_0)}) \ge \epsilon_2\}\cap[\tilde M_i,\tilde N_i-1]|}{\tilde N_i-\tilde M_i}=1.$$
			This ends the proof of Claim \ref{c-2}.
		\end{proof}
	Let $\{\tilde N_i\}_{i=1}^\infty,\{\tilde M_i\}_{i=1}^\infty$ be as in Claim \ref{c-2}. Passing by a subsequence, we can assume that $\{[\tilde M_i,\tilde N_i-1]\}_{i\in \N}$ is a tempered F{\o}lner sequence. By Claim \ref{c-1}, $\nu_e( \mathcal{C}(x_0,\gamma_0))>0$. By Corollary \ref{cor-1}, there is a generic point $E_0\in \mathcal{C}(x_0,\gamma_0)$ such that 
$$\nu_e=\lim_{i\to\infty}\frac{1}{\tilde N_i-\tilde M_i}\sum_{n=\tilde M_i}^{\tilde N_i-1}\delta_{T^nE_0}.$$
Then by Claim \ref{c-2}, $\nu_e(\mathcal{E}_{\ge \epsilon_2})=1$. This conflicts with \eqref{1} and then $\pi_{eq}$ is almost $1$-$1$.
	\end{proof}
\begin{rem}Same argument shows that Theorem \ref{thm-B} holds for all minimal countable abelian group actions on compact metric space. It is not clear whether Theorem \ref{thm-B}  holds for minimal countable amenable group actions  on compact metric space.
\end{rem}

 \noindent{\bf Acknowledgment.} 
L. Xu is partially supported by NNSF of China (12031019, 12371197).

\end{document}